\documentclass[final, times]{elsarticle}
\usepackage[english]{babel}
\usepackage{amsmath,amsthm}
\usepackage{amsfonts}
\usepackage[bookmarks=true, pdfstartview={FitH}]{hyperref}
\usepackage[margin=1in]{geometry}
\newtheorem{thm}{Theorem}[section]
\newtheorem{thma}{Theorem}

\newtheorem{cora}{Corollary}

\newtheorem{lem}{Lemma}[section]
\newtheorem{lema}{Lemma}

\newtheorem{prop}{Proposition}[section]
\newtheorem{propa}{Proposition}

\theoremstyle{definition}
\newtheorem{defn}{Definition}[section]
\theoremstyle{remark}
\newtheorem{rem}[thm]{Remark}
\numberwithin{equation}{section}

\journal{choose}
\begin{document}
\begin{frontmatter}
\title{A Characterization of the Two-weight Inequality for Riesz Potentials on Cones of Radially Decreasing Functions}
\author[add1,add2]{Alexander Meskhi}
\ead{alex72meskhi@yahoo.com; meskhi@rmi.ge}
\author[add3]{Ghulam Murtaza}
\ead{gmnizami@googlemail.com}
\author[add4]{Muhammad Sarwar}
\ead{sarwar@uom.edu.pk}

\address[add1]{Department of Mathematical Analysis, A. Razmadze Mathematical Institute, I. Javakhishvili Tbilisi State University, 2. University Str., 0186 Tbilisi, Georgia.}
\address[add2]{Department of Mathematics,  Faculty of Informatics and Control Systems, Georgian Technical University, 77, Kostava St., Tbilisi, Georgia.}
\address[add3]{Department of Mathematics,  GC University, Faisalabad, Pakistan.}
\address[add4]{ Department of Mathematics University of Malakand, Chakdara, Dir(L), Khyber Pakhtunkhwa, Pakistan}

\begin{abstract}
We establish necessary and sufficient conditions on a weight pair $(v,w)$ governing the boundedness of the Riesz potential operator $I_{\alpha}$ defined on a homogeneous group $G$ from $L^p_{dec,r}(w, G)$ to $L^q(v, G)$, where $L^p_{dec,r}(w, G)$ is the Lebesgue space defined for non-negative radially decreasing functions on $G$. The same problem is also studied for the potential operator with product kernels $I_{\alpha_1, \alpha_2}$ defined on a product of two homogeneous groups $G_1\times G_2$. In the latter case weights, in general, are not of product type. The derived results  are new even for Euclidean spaces. To get the main results we use Sawyer type duality theorems (which are also discussed in this paper) and two--weight Hardy type inequalities on $G$ and $G_1\times G_2$ respectively.
\end{abstract}

\end{frontmatter}

\newpage

\section{Introduction}
A homogeneous group is a simply connected nilpotent Lie
group G on a Lie algebra g with the one-parameter group of
transformations $\delta_{t}=exp(A\ log\ t)$, $t>0$, where A is a
diagonalized linear operator in $G$ with positive eigenvalues. In
the homogeneous group $G$ the mappings $exp\ o \ \delta_t\ o\
exp^{-1}$, $t>0$, are automorphisms in G, which will be again
denoted by $\delta_{t} $. The number $Q=tr\ A $ is the homogeneous
dimension of $G$. The symbol $e$ will stand for the neutral element in
$G$.
\par
It is possible to equip $G$ with a homogeneous norm $ r:G\rightarrow [
\ 0,\infty)$ which is continuous on $G$, smooth on $G\backslash\{e\}$
and satisfies the conditions:

(i) \ \ \  $r(x)=r(x^{-1})$  for every $x\in \ G$;

(ii)\ \ \ $r(\delta_t x)= tr(x) $ for every $x\in G$ and $t>0$;

(iii)\ \ $ r(x) = 0$ if and only if  $x=e$ ;

(iv) \  There exists $ c_o >0 $ such that
$$  r(xy) \leq c_o (r(x)+r(y)),\;\; x,y  \in  G. $$

In the sequel we denote by $B(a,t)$ an open
ball with the center $a$ and radius
$t>0$, \emph{i.e.}\\
$$ B(a,t ) : = \{ y \in \ G; \ r(ay^{-1}) < t  \}.$$

It can be observed that $ \delta _ t B(e,1) =B(e, t).$

Let us fix a Haar measure $|\cdot|$ in G such that $|B(e,1)|=1$. Then
$|\delta_{t} E|=t^Q|E|$. In particular, $|B(x,t)|=t^Q \hbox { for }x
\ \in \ G , \ t>0$.

Examples of homogeneous groups are: the Euclidean n-dimensional
space $\mathbb{R}^{n}$,  the Heisenberg   group, upper triangular
groups, etc. For the definition and basic properties of the
homogeneous  group we refer to  \cite{FS}, p. 12.

\vskip+0.2cm

\indent An everywhere positive function $\rho$ on $G$ will be called a weight. Denote by $L^p(\rho, G) \ (1<p<\infty)$ the
weighted Lebesgue space, which is the space of all measurable
functions $f:G\rightarrow \mathbb{C}$ defined by the norm
$$\|f\|_{L^p(\rho, G)}=\Big(\int\limits_G
|f(x)|^p \rho(x)dx\Big)^{\frac{1}{p}}<\infty. \nonumber $$ If $\rho\equiv
1$, then    we  we use the notation  $L^p(G).$

Denote by ${\mathcal{DR}}(G)$ the class of all radially decreasing functions on $G$ with values in ${\Bbb{R}}_+$, i.e. the fact that $\phi \in {\mathcal{DR}}(G)$ means that there is decreasing $\bar{\phi}: {\Bbb{R}}_+ \mapsto {\Bbb{R}}_+$ such that $\varphi(x)= \bar{\phi}(r(x))$. In the sequel we will use the symbol $\phi$ itself for $\bar{\phi}$; the fact that $\phi \in {\mathcal{DR}}(G)$ will be written also by the symbol $\varphi \downarrow r$. Let $G_1$ and $G_2$ be homogeneous groups. We say that a function $\psi: G_1\times G_2\mapsto {\Bbb{R}}_+$ is radially decreasing if it is such in each variable separately uniformly to another one. The fact that $\psi$ is radially decreasing on $G_1\times G_2$  will be denoted as $\psi \in {\mathcal{DR}}(G_1\times G_2)$.

Let $$(I_{\alpha}f)(x) = \int\limits_G f(y) \big(r(xy^{-1})\big)^{\alpha-Q} dy, \;\;\;\; 0<\alpha <Q, $$
be the Riesz potential defined on $G$, where $r$ is the homogeneous norm and $dy$ is the normalized Haar measure on $G$. The operator $I_{\alpha}$ plays a fundamental role in harmonic analysis, e.g.,  in the theory of Sobolev embeddings, in the theory of sublaplacians on nilpotent groups etc. Weighted estimates for multiple Riesz potentials can be applied, for example, to establish Sobolev and Poincar\'e inequalities on product spaces (see, e.g., \cite{ToPr}).

Let  $G_1$ and $G_2$ be homogeneous groups with homogeneous norms $r_1$ and $r_2$ and homogeneous dimensions $Q_1$ and $Q_2$ respectively. We define the potential operator on $G_1\times G_2$ as follows
$$ I_{\alpha, \beta} f(x,y) = \iint\limits_{G_1\times G_2} f(t,\tau) \big(r_1(xt^{-1})\big)^{\alpha-Q_1} \big(r_2(y\tau^{-1})\big)^{\beta-Q_2}dtd\tau, \;\;\; (x,y)\in G_1\times G_2, \;\; 0<\alpha<Q_1,\; 0<\beta<Q_2. $$

Our aim is to derive two-weight criteria for $I_{\alpha}$ on the cone of radially decreasing functions on $G$. The same problem is also studied for the potential operator with product kernels $I_{\alpha, \beta}$ defined on a product of two homogeneous groups, where  only the right--hand side weight is of product type. As far as we know the derived results for $I_{\alpha, \beta}$ are new even in the case of  Euclidean spaces. The proofs of the main results are based on E. Sawyer (see \cite{Saw})  type  duality theorem which is also true for homogeneous groups (see Propositions  \ref{Saw-Dua} and  \ref{BHP2} below) and  Hardy type two-weight inequalities in homogeneous groups.
Analogous results for multiple potential  operators defined  on ${\Bbb{R}}^n_+$ with respect to the cone of non-negative decreasing functions on ${\Bbb{R}}^n_+$ were studied in \cite{MMS}, \cite{MM}. It should be emphasized that the two-weight problem for multiple Hardy operator for the cone of decreasing functions on ${\Bbb{R}}^n_+$ was investigated by S. Barza, H. P. Heinig and L. -E. Persson \cite{BHP} under the restriction that both weights are of product type.

Historically the one-weight inequality for the classical Hardy operator was characterize by M. A. Arino and B. Muckenhoupt \cite{AM} under the so called  $B_p$ condition. The same problem for multiple Hardy transform was studied by N. Arcozzi,  S. Barza, J. L.  Garcia-Domingo  and J. Soria \cite{ABGS}. This problem in the the two-weight setting was solved by E. Sawyer \cite{Saw}. Some sufficient conditions guaranteeing  the two--weight inequality for the Riesz potential $I_{\alpha}$ on ${\Bbb{R}}^n$ was given by Y. Rakotondratsimba \cite{Rak}. In particular, the author showed that $I_{\alpha}$ is bounded from $L^p_{dec,r}(w, {\Bbb{R}}^n)$ to $L^q(v, {\Bbb{R}}^n)$ if the weighted  Hardy operators $({\mathcal{H}}f)(x)=\frac{1}{|x|^{n-\alpha}} \int\limits_{|y|<|x|} f(y) dy $ and $({\mathcal{H}}'f)(x)= \int\limits_{|y|>|x|} \frac{f(y)}{|y|^{n-\alpha}} dy$ are bounded from $L^p(w, {\Bbb{R}}^n)$ to $L^q(v,{\Bbb{R}}^n)$. In fact the author studied  the problem on the cone of monotone decreasing functions.

Now we give some comments regarding the notation: in the sequel under the symbol $A\approx B$ we mean that there are positive  constants $c_1$ and $c_2$  (depending on appropriate parameters) such that $ c_1A \leq B \leq c_2 A$;  $A \ll B$ means that there is a positive constant $c$ such that $A \leq  c B$; integral over a product set $E_1\times E_2$ from $g$ will be denoted by $\iint\limits_{E_1\times E_2} g(x,y) dx dy$ or $\int_{E_1} \int_{E_2} g(x,y) dx dy$; for a weight functions $w$ and $w_i$  on $G$, by the symbols $W(t)$ and $W_i(t)$ will be denoted the integrals  $\int\limits_{B(e,t)} w(x)dx $ and $\int\limits_{B(e_i,t)} w_i(x)dx $ respectively; for a weight $w$ on $G_1\times G_2$, we denote  $W(t,\tau):= \int\limits_{B(e_1,t)\times B(e_2,\tau)} w(x,y) dx dy$, where $e_1$ and $e_1$ are neutral elements in $G_1$ and $G_2$ respectively. Finally we mention that constants (often different constants in one and the same lines of inequalities) will be denoted by $c$ or $C$. The symbol $p'$ stands for  the conjugate number of $p$:  $p'=p/(p-1)$, where $1<p<\infty$.

\section{Preliminaries}

We begin this section with the statements regarding polar coordinates in $G$ (see e.g., \cite{FS}, P. 14).
\vskip+0.2cm
\begin{propa}\label{polar}  Let $G$ be a homogeneous group and let
$S=\{x\in G:r(x)=1\}$. There is a (unique) Radon measure $\sigma$ on
$S$ such that for all $u\in L^1(G)$,
$$\int \limits_G u(x)dx=\int\limits_0^\infty \int\limits_S
u(\delta_t\overline y)t^{Q-1}d\sigma(\overline y)dt.$$
\end{propa}

Let $a$ be a positive number. The two--weight inequality for the Hardy-type transforms
$$ (H^{a}f)(x)= \int\limits_{B(e, a r(x))} f(y) dy, \;\;\;\; x\in G, $$

$$ (\widetilde{H}^{a}f(x) = \int\limits_{G\setminus B(e, a r(x))} f(y) dy, \;\;\;\; x\in G, $$
reeds as follows (see \cite{EKM}, Ch.1 for more general case, in particular for quasi-metric measure spaces):

\begin{thma}\label{Hardy-G}  Let $1<p\leq q<\infty$ and let $a$ be a positive number. Then

\rm{(i)}

The operator $H^a$ is bounded from $L^p(u_1, G)$ to $L^q(u_2,G)$ if and only if
$$ \sup_{t>0}  \bigg( \int_{G\setminus B(e,t)} u_2(x) dx \bigg)^{1/q} \bigg( \int\limits_{B(e,at)} u_1^{1-p'}(x) dx \bigg)^{1/p'}<\infty. $$

\rm{(ii)}

The operator $\widetilde{H}^a$ is bounded from $L^p(u_1, G)$ to $L^q(u_2,G)$ if and only if
$$ \sup_{t>0}  \bigg( \int_{B(e,t)} u_2(x) dx \bigg)^{1/q} \bigg( \int\limits_{G\setminus B(e,at)} u_1^{1-p'}(x) dx \bigg)^{1/p'}<\infty. $$

\end{thma}
We refer also to \cite{DHK} for the Hardy inequality written for balls with center at the origin.
\vskip+0.2cm

In the sequel we denote $H^{1}$ by $H$.

\vskip+0.2cm

The following statement for Euclidean spaces was derived by S. Barza, M. Johansson and L. -E. Persson \cite{BJP}.
\begin{propa}\label{Duality} Let $w$ be a weight function on $G$ and let $1<p<\infty$. If $f\in {\mathcal{DR}}(G)$ and $g \geq 0$, then
$$ \sup_{f \downarrow r} \frac{\int\limits_{G} f(x)g(x) dx  }{\Big(\int\limits_G f(x)^p w(x) dx \Big)^{1/p}} \approx \|w\|_{L^1(G)}^{-1/p} \|g\|_{L^1(G)}+ \bigg( \int\limits_{G} H^{p'}(r(x)) W^{-p'}(r(x)) w(x) dx \bigg)^{1/p'}, $$
where $H(t)=\int\limits_{B(e,t)}g(x) dx$, $W(t)=\int\limits_{B(e,t)}w(x) dx$.
\end{propa}

The proof of  Proposition \ref{Duality} repeats the arguments (for ${\Bbb{R}}^n$) used in the proof of Theorem 3.1 of \cite{BJP} taking  Proposition \ref{polar} and the following lemma into account.

\begin{lema} let $1<p<\infty$. For a weight function $w$, the inequality
$$ \int\limits_G w(x) \bigg( \int\limits_{G\setminus B(e, r(x))} f(y) dy\bigg)^p dx \leq p \int\limits_G f^p(x)  W^p(r(x)) w^{1-p}(x) dx , \;\;\;\; f\geq 0,$$
holds.
\end{lema}
{\em Proof} of this lemma is based on  Theorem \ref{Hardy-G} (part (ii)) taking $a=1$, $p=q$, $u_2(x)= v(x)$, $u_1= w^{1-p}(x) W^p(r(x))$ there.  Details are omitted. $\;\;\; \Box$

\vskip+0.2cm

\begin{cora}\label{Duality-1} Let the conditions of Proposition \ref{Duality} be satisfied and let $\int\limits_{G}w(x) dx=\infty$. Then
the following relation holds:
$$ \sup_{f \downarrow r} \frac{\int\limits_{G} f(x)g(x) dx  }{\Big(\int\limits_G f^p(x) w(x)dx \Big)^{1/p}} \approx  \bigg( \int\limits_{G} H^{p'}(r(x)) W(r(x)) w(x) dx \bigg)^{1/p}.$$
\end{cora}

Corollary  \ref{Duality-1} implies the following duality result which follows by the standard way (see \cite{Saw}, \cite{BJP} for details).
\begin{propa}\label{Saw-Dua} Let $1< p,q<\infty$ and let $v,w$ be weight functions on $G$ with $\int\limits_{G} w(x) dx =\infty$. Then the integral operator $T$ defined on functions on $G$ is bounded from $L^p_{dec,r}(w, G)$ to $L^q(v, G)$ if and only if
\begin{equation}
\bigg( \int\limits_{G} \bigg( \int\limits_{B(e, r(x))} (T^*g)(y) dy \bigg)^{p'} W^{-p'}(r(x)) w(x) dx \bigg)^{1/p'} \leq C \bigg( \int\limits_{G} g^{q'}(x) v^{1-q'}(x) dx \bigg)^{1/q'}
\end{equation}
holds for every positive measurable $g$ on $G$.
\end{propa}

The next statement yields the criteria for the two--weight boundedness of the  operator $H$ on the cone ${\mathcal{DR}}(G)$. In particular the following statement is true:

\begin{thma}\label{Hardy-Inequality-G} Let $1<p\leq q <\infty$ and let $v$ and $w$ be weights on $G$ such that $\|w\|_{L^1(G)}=\infty$. Then $H$ is bounded from $L^p_{dec,r}(w,G)$ to $L^q_{v}(v,G)$ if and only if

\rm{(i)}
$$ \sup_{t>0} \bigg( \int\limits_{B(e,t)} w(x) dx \bigg)^{-1/p} \bigg( \int\limits_{B(e,t)} v(x) r^{Qq}(x) dx\bigg)^{1/q}<\infty; $$
\rm{(ii)}
$$ \sup_{t>0} \bigg( \int\limits_{B(e,t)} r^{Qp'}(x) W^{-p'}(r(x)) w(x) dx \bigg)^{1/p'} \bigg( \int\limits_{G\setminus B(e,t)} v(x) dx \bigg)^{1/q} <\infty. $$
\end{thma}

{\em Proof} of this statement follows by the standard way applying Proposition  \ref{Saw-Dua} (see e.g. \cite{Saw}, \cite{BJP}). $\Box$.
\vskip+0.2cm




\begin{defn} Let $\rho$ be a locally integrable a.e.
positive function on $G$. We say that $\rho$ satisfies the  doubling
condition at $e$ ( $\rho \in DC(G)$ )
if there is a positive constant $b>1$ such that for all $t>0$ the following inequality holds:

$$  \int\limits_{B(e, 2t)} \rho(x) dx \leq b    \int\limits_{B(e,t)} \rho(x)  dx.$$

Further, we say that $w\in  DC^{\gamma, p}(G)$, where $1<p<\infty$, $0<\gamma<Q/p$, if there is a positive constant $b$ such that for all $t>0$
$$  \int\limits_{G\setminus B(e, t)} r^{\gamma p'}(x) W^{-p'}(r(x))w(x)  dx \leq b    \int\limits_{G\setminus B(e,2t)} r^{\gamma p'}(x) W^{-p'}(r(x))w(x)  dx.$$
\end{defn}

\begin{rem}\label{Rem} It is also to check that under the assumption  $1<p<\infty$, $0<\gamma<Q/p$ the condition $w\in  DC^{\gamma, p}(G)$ is satisfied for $w\equiv \; const$.
\end{rem}

\begin{defn} We say that a locally integrable a.e. positive function $\rho$ on $G_1\times G_2$ satisfies the  doubling
condition with respect to the second variable ( $\rho \in DC^{(s)}(y)$ )
uniformly to the first one if there is a positive constant $c$ such that for all $t>0$ and
almost every $x\in G_1$ the following inequality holds:

$$\int\limits_{B(e_2, 2t)} \rho(x,y) dy \leq c  \int\limits_{B(e_2,t)} \rho(x,y)  dy.$$

Analogously is defined the class of weights $DC^{(s)}(x)$.
\end{defn}

\section{Riesz Potentials on $G$}

The main result of this section reeds as follows:

\begin{thm}\label{Main-Theorem} Let $1<p\leq q<\infty$ and let $v$ and $w$ be weights such that either $w\in DC^{\alpha, p}(G) $ or $v\in DC(G)$; let $\|w\|_{L^1(G)}=\infty$. Then the operator $I_{\alpha}$ is
bounded from $L^p_{dec,r}(w,G)$ to $L^q(v, G)$ if and only if

\rm{(i)}
\begin{equation}\label{F1}
\sup_{t>0} \bigg(\int\limits_{B(e,t)}  w(x) dx \bigg)^{-1/p}  \bigg(\int\limits_{B(e,t)}   r^{\alpha q}(x) v(x) dx \bigg)^{1/q}< \infty;
\end{equation}

\rm{(ii)}
\begin{equation}\label{F2}
\sup_{t>0} \bigg(\int\limits_{B(e,t)}  r^{p'Q}(x)W^{-p'}(r(x)) w(x)
dx \bigg)^{1/p'} \bigg(\int_{G\setminus B(e, t)}   r^{(\alpha-Q)q}(x) v(x)dx \bigg)^{1/q} <\infty;
\end{equation}

\rm{(iii)}
\begin{equation}\label{F3}
\sup_{t>0} \bigg(\int\limits_{B(e,t)}  v(x) dx \bigg)^{1/q} \bigg(\int_{G\setminus B(e, t)} r^{\alpha p'}(x) W^{-p'}(r(x)) w(x) dx \bigg)^{1/p'} <\infty.
\end{equation}

\end{thm}

To prove this result we need to prove some auxiliary statements.

\begin{lem}\label{estimate}
Let $0<\alpha<Q$ and let $c_o$ be the constant from the triangle inequality of $r$. Then there is a positive constant $c$ depending only on $Q$, $\alpha$ and $c_o$ such that for all $s\in B(e, r(x)/2)$,
\begin{equation}\label{est}
I(x,y):= \int\limits_{B(e, r(x))\setminus B(e, 2c_0 r(y))} r(ty^{-1})^{\alpha-Q} dt \leq c r(xy^{-1})^{\alpha}.
\end{equation}
\end{lem}
\begin{proof}
We have
$$ I(x,y)= \int\limits_0^{\infty} |\{ t\in G: r(ty^{-1})^{\alpha-Q}>\lambda\}\cap B(e, r(x))\setminus B(e, 2c_0 r(y))|d \lambda = \int\limits_{0}^{r(xy^{-1})^{\alpha-Q}} (\cdots) + \int\limits_{r(xy^{-1})^{\alpha-Q}}^{\infty}(\cdots) =: I^{(1)}(x,y) + I^{(2)}(x,y). $$
Observe that, by the triangle inequality for $r$, we have $r^Q(x)\leq c^{Q}_0 2^{Q-1}( r^Q(xy^{-1}) + r^Q(y))$. This implies that $r^Q(x)- (2c_0)^Q r^Q(y) \leq c^{Q}_0 2^{Q-1} r^Q(xy^{-1})$. Hence,
$$ I^{(1)}(x,y) \leq r(xy^{-1})^{\alpha-Q} |B(e, r(x))\setminus B(e, 2c_0 r(y))| =r(xy^{-1})^{\alpha-Q} \Big( r^Q(x)- (2c_0)^Q r^Q(y)\Big) \leq c r(xy^{-1})^{\alpha}.$$
Further, it is easy to see that
$$ I^{(2)}(x,y) \leq c r(xy^{-1})^{\alpha}.$$

Finally we have \eqref{est}.
\end{proof}

Let us introduce the following potential operators
$$ (J_{\alpha}f)(x)= \int\limits_{B(e,2c_0 r(x))}  f(y) r^{\alpha-Q}(xy^{-1}) dy, \;\;\; (S_{\alpha}f)(x)= \int\limits_{G\setminus B(e,2c_0 r(x))}  f(y) r^{\alpha-Q}(xy^{-1}) dy,\;\;\;\ x\in G, \; 0<\alpha<Q. $$

It is easy to see that

\begin{equation}\label{representation}
 I_{\alpha}f = J_{\alpha}f + S_{\alpha}f.
\end{equation}

We need also to introduce the following weighted Hardy operator
$$ (H_{\alpha}f)(x)= r(x)^{\alpha-Q} (Hf)(x). $$

\begin{prop}\label{main1}
 The following relation holds for all $f\in {\mathcal{DR}}(G)$
\begin{equation}
J_{\alpha}f \approx H_{\alpha}f.
\end{equation}
\end{prop}
\begin{proof}
We have

$$ (J_{\alpha}f)(x) = \int\limits_{B(e, r(x)/ 2c_0)} f(y) r^{\alpha-Q}(xy^{-1}) dy  + \int\limits_{B(e, 2c_0 r(x))\setminus B \big(e, r(x)/(2c_0)\big)} f(y) r^{\alpha-Q}(xy^{-1}) dy=: (J^{(1)}_{\alpha}f)(x) + (J^{(2)}_{\alpha}f)(x). $$
If $y\in B(e, r(x)/ 2c_0)$, then $r(x) \leq c_0 (r(xy^{-1}) +r(y)) \leq c_0 r(xy^{-1}) + r(x)/2 $. Hence $r(x) \leq 2 c_0 (r(xy^{-1})$. Consequently,
$$(J^{(1)}_{\alpha}f)(x) \leq c (H_{\alpha}f)(x).$$
Applying now the fact that $f\in DR (G)$ we see that
\begin{eqnarray*} (J^{(2)}_{\alpha}f)(x)&\leq& f(r(x)/2c_0) \int\limits_{B(e, r(x)/ 2c_0)\setminus B(e, 2c_0 r(x))} r^{\alpha-Q}(xy^{-1}) dy \leq c f(r(x)/2c_0) r(x)^{\alpha} \leq c (H_{\alpha}f)(x).
\end{eqnarray*}
\end{proof}

\begin{lem}\label{main2}Let $1<p\leq q<\infty$ and let $v$ and $w$ be weights on $G$ such that $\| w \|_{L^1(G)}=\infty$. Then the operator $S_{\alpha}$ is bounded from $L^p_{dec,r}(w,G)$ to $L^q(v,G)$ if

$$ \sup_{t>0}  \bigg( \int\limits_{G\setminus B(e,t)} r^{\alpha p'}(x) W^{-p'}(r(x))w(x) dx \bigg)^{1/p'} \bigg( \int\limits_{B\big(e,t/(2c_0)\big)}  v(x) dx \bigg)^{1/q}<\infty. $$\
Conversely, if $S_{\alpha}$ is bounded from $L^p_{dec,r}(w,G)$ to $L^q(v,G)$, then the condition
$$ \sup_{t>0}  \bigg( \int_{G\setminus B(e,t)} r^{\alpha p'}(x) W^{-p'}(x)w(x) dx \bigg)^{1/q} \bigg( \int\limits_{B\big(e,t/(4c_0)\big)}  v(x) dx \bigg)^{1/p'}<\infty $$
is satisfied. Furthermore, if either $w\in DC^{\alpha, p}$ or $v\in DC(G)$, then  the operator $S_{\alpha}$ is bounded from $L^p_{dec,r}(w,G)$ to $L^q(v,G)$ if and only if
$$ \sup_{t>0}  \bigg( \int_{G\setminus B(e,t)} r^{\alpha p'}(x) W^{-p'}(r(x))w(x) dx \bigg)^{1/q} \bigg( \int\limits_{B(e,t)} v(x) dx \bigg)^{1/p'}<\infty. $$
\end{lem}
\begin{proof}
Applying Proposition \ref{Saw-Dua}, $S_{\alpha}$ is bounded from $L^p_{dec,r}(w,G)$ to $L^q(v,G)$ if and only if
$$ \bigg( \int\limits_G \bigg( \int\limits_{B(e, r(x))}  (S^*_{\alpha}f)(y) dy\bigg)^{p'} W^{-p'}(r(x)) w(x)dx \bigg)^{1/p'} \leq c\bigg( \int\limits_G g^{q'}(x) v^{1-q'}(x) dx \bigg)^{1/q'}, $$
where
$$ (S^*_{\alpha}f)(x)= \int\limits_{B\big(e, r(x)/(2c_0)\big)}  f(y) r^{\alpha-Q}(xy^{-1}) dy.$$

Now we show that
\begin{equation}\label{two-sided}
c_1 r^{\alpha}(x) \int\limits_{B\big(e, r(x)/ (4c_0)\big)} g(s) ds \leq  \int\limits_{B(e, r(x))}  (S^*_{\alpha}g)(y) dy   \leq c_2 r^{\alpha}(x) \int\limits_{B\big(e, r(x)/ (2c_0)\big)}g(s) ds,\;\;\; g\geq 0.
\end{equation}
To prove the right-hand side estimate in \eqref{two-sided} observe that by Tonelli's theorem and  Lemma \ref{estimate} we have that
\begin{eqnarray*}
\int\limits_{B(e, r(x))}  (S^*_{\alpha}g)(y) dy &=& \int\limits_{B\big(e, r(x)/(2c_0)\big)} f(s)  \bigg( \int\limits_{B\big(e, r(x)\big) \setminus B \big(e, 2c_0 r(s)\big)} r^{\alpha-Q}(sy^{-1}) dy \bigg) ds \\
 &\leq &   c_2 r(x)^{\alpha}  \int\limits_{B(e, r(x)/(2c_0))} f(s) ds .
\end{eqnarray*}
On the other hand,

\begin{eqnarray*}
\int\limits_{B(e, r(x))}  (S^*_{\alpha}g)(y) dy &\geq& c r^{\alpha-Q}(x) \bigg( \int\limits_{B(e, r(x))\setminus B(e, r(x)/2)}  \bigg( \int\limits_{B\big(e, r(y)/(2c_0)\big)} f(s) ds \bigg) dy\bigg)  \\
 &\geq& c_1 r^{\alpha}(x) \bigg( \int\limits_{B\big(e,   r(x)/(4c_0)\big) } f(s) ds   \bigg) .
\end{eqnarray*}
Thus, Theorem \ref{Hardy-G} completes the proof.
\end{proof}

{\em Proof} of Theorem \ref{Main-Theorem}.  By \eqref{representation} it is enough to estimate the terms with
$J_{\alpha}f$ and $S_{\alpha}f. $
By applying  Proposition \ref{main1} and Theorem \ref{Hardy-Inequality-G} we have that $J_{\alpha}$ is bounded from $L^p_{dec,r}(w,G)$ to $L^q(v,G)$ if and only if the conditions (ii) and (iii) are satisfied. Now by Lemma \ref{main2} and the equality (which is a consequence of Proposition \ref{polar})
 $$\bigg( \int\limits_{G\setminus B(e,t)} W(r(x)) w(x) dx\bigg)^{1/p'}= \bigg(\int\limits_{B(e,t)}  w(x) dx \bigg)^{-1/p} $$
 we have that $S_{\alpha}$  is bounded from $L^p_{dec,r}(w,G)$ to $L^q(v,G)$ if and only if (i) is satisfied. $\Box$

\section{Multiple Potentials on $G_1 \times G_2$}

Let us now investigate the two--weight problem for the operator $I_{\alpha, \alpha_2}$ on the cone ${\mathcal{DR}}(G_1\times G_2)$. In the sequel without loss of generality we denote  the triangle inequality constants for $G_1$ and $G_2$ by one and the same symbol $c_0$.

The following statement can be derived just in the same way as Theorem 3.1 was obtained  in \cite{BHP}. The proof is omitted because to avoid repeating  those arguments.
\vskip+0.2cm
\begin{propa}\label{BHP1} Let $1<p<\infty$ and let $w(x,y)= w_1(x)w_2(y)$ be a product weight on $G_1\times G_2$. Then the following relation
$$\sup_{0\leq f \downarrow r}\frac{\iint\limits_{G_1\times G_2} f(x,y)g(x,y) dxdy}{ \bigg(\iint\limits_{G_1\times G_2} f^{p}(x,y) w(x,y)\bigg)^{1/p}} \approx \sum_{i=1}^4 I_k,$$
holds for a non-negative measurable function $g$, where
$$ I_1:= \|w\|_{L^1(G_1\times G_2)}^{-1/p} \|g\|_{L^1(G_1 \times G_2)}, $$

$$  I_2:= \| w_2\|^{-1/p}_{L^1(G_1)} \bigg( \int\limits_{G_1} \int\limits_{B(e_1,r_1(x))} \| g(t,\cdot)\|_{L^1(G_2)}dt \bigg)^{p'} W_1^{-p'}(r_1(x))w_1(x) dx\bigg)^{1/p'}, $$

$$  I_2:= \| w_1\|^{-1/p}_{L^1(G_1)} \bigg( \int\limits_{G_2} \int\limits_{B(e_2,r_2(y))} \| g(\cdot, \tau)\|_{L^1(G_1)}d\tau \bigg)^{p'} W_2^{-p'}(r_2(y))w_2(y) dy\bigg)^{1/p'}, $$

$$ I_4:=\bigg( \int\limits_{G_1\times G_2} \bigg( \int\limits_{G_1\times G_2} g(t,\tau) dt d\tau \bigg)^{p'} W^{-p'}(r_1(x),r_2(y))w(x,y) dx dy\bigg)^{1/p'}.$$
\end{propa}

Applying  Proposition \ref{BHP1}  together with the duality arguments we can get the following statement (cf. \cite{BHP}).
\begin{propa}\label{BHP2}
Let $1<p<\infty$ and let  $v$ and $w$ be weights on $G_1\times G_2$ such that $w(x,y)= w_1(x) w_2(y)$, $\|w\|_{L^1(G_1\times G_2)}=\infty$. Then an integral operator $T$ defined for  functions from  ${\mathcal{DR}}(G_1\times G_2)$ is bounded from $L^p_{dec, r}(w, G_1\times G_2)$ to $L^p(v, G_1\times G_2)$ if and only if for all non-negative measurable $g$ on $G_1 \times G_2$,
$$ \bigg( \iint\limits_{G_1\times G_2} \bigg( \iint\limits_{B(e_1, r_1(x))\times B(e_2, r_2(y))} (T^*g)(t,\tau)dt d\tau\bigg)^{p'} W^{-p'}(x,y)  w(x,y) dxdy\bigg)^{1/p'} \leq C \bigg( \iint\limits_{G_1\times G_2} g^{q'}(x,y) v^{1-q'}(x,y) dx dy\bigg)^{1/q'}.$$
\end{propa}

The next statements deals with the double Hardy--type operators defined on $G_1\times G_2$

$$ (H^{a,b}f)(x,y)= \int\limits_{B(e_1, a r_1(x))} \int\limits_{B(e_2, b r_2(x)) } f(t, \tau) dt d\tau, \;\;\;\; (x,y)\in G_1 \times G_2, $$

$$ (\tilde{H}^{a,b}f)(x,y)= \int\limits_{G_1\setminus B(e_1, a r_1(x))} \int\limits_{G_2\setminus B(e_2, b r_2(x))} f(t, \tau) dt d\tau, \;\;\;\; (x,y)\in G_1 \times G_2, $$

$$
(H_1^{a,b}f)(x,y)= \int\limits_{B(e_1, a r_1(x))}  \int\limits_{G_2 \setminus B(e_2, b r_2(y))}   f(t, \tau) dt d\tau, \;\;\;\; (x,y)\in G_1 \times G_2, $$

$$
(H_2^{a,b}f)(x,y)= \int\limits_{G_1\setminus B(e_1, a r_1 (x))}  \int\limits_{B(e_2, b r_2(y))}   f(t, \tau) dt d\tau, \;\;\;\; (x,y)\in G_1 \times G_2. $$

\begin{prop}\label{Hardy-G_1-G_2}  Let $1<p\leq q<\infty$. Suppose that $v$ and $w$ be weights on $G_1\times G_2$ such that either $w(x,y)= w_1(x) w_2(y)$ or $v(x,y)= v_1(x) v_2(y)$. Then

 \rm{(i)} The operator  $H^{a,b}$ is bounded from $L^p(w, G_1\times G_2)$ to $L^q(v, G_1\times G_2)$ if and only if
$$ A:= \sup_{t>0, \tau>0}  \bigg( \int\limits_{G_1\setminus B(e_1,t)}\int\limits_{G_2\setminus B(e_2,\tau)} v(x,y) dxdy \bigg)^{1/q} \bigg( \int\limits_{B(e_1,at)} \int\limits_{B(e_2,b\tau)} w^{1-p'}(x,y) dxdy \bigg)^{1/p'}<\infty. $$

\rm{(ii)} The operator  $\tilde{H}^{a,b}$ is bounded from $L^p(w, G_1\times G_2)$ to $L^q(v, G_1\times G_2)$ if and only if
$$ \sup_{t>0, \tau>0}  \bigg( \int\limits_{B(e_1,t)} \int\limits_{B(e_2,\tau)} v(x,y) dxdy \bigg)^{1/q} \bigg(  \int\limits_{G_1\setminus B(e_1,at)}\int\limits_{G_2\setminus B(e_2, b \tau)} w^{1-p'}(x,y) dxdy \bigg)^{1/p'}<\infty. $$

\rm{(iii)} The operator  $H_1^{a,b}$ is bounded from $L^p(w, G_1\times G_2)$ to $L^q(v, G_1\times G_2)$ if and only if
$$ \sup_{t>0, \tau>0}  \bigg( \int\limits_{G_1\setminus B(e_1,t)} \int\limits_{B(e_2, \tau)} v(x,y) dxdy \bigg)^{1/q} \bigg(  \int\limits_{B(e_1,at)}\int\limits_{G_2\setminus B(e_2, b \tau)} w^{1-p'}(x,y) dxdy \bigg)^{1/p'}<\infty. $$

\rm{(iv)} The operator  $H_2^{a,b}$ is bounded from $L^p(w, G_1\times G_2)$ to $L^q(v, G_1\times G_2)$ if and only if
$$ \sup_{t>0, \tau>0}  \bigg( \int\limits_{B(e_1,t) } \int\limits_{G_2\setminus B(e_2,\tau )} v(x,y) dxdy \bigg)^{1/q} \bigg(  \int\limits_{G_1\setminus B(e_1,at)}\int\limits_{ B(e_2, b \tau)} w^{1-p'}(x,y) dxdy \bigg)^{1/p'}<\infty. $$
\end{prop}

\begin{proof}  Let $w(x,y)= w_1(x) w_2(y)$. Then the proposition  follows  in the same way as the appropriate statements regarding the Hardy operators defined on ${\Bbb{R}}_+^2$ in \cite{MeJFSA}, \cite{KoMe3} (see also Theorem 1.1.6 of \cite{KMP}). If $v$ is  a product weight, i.e.  $v(x,y)= v_1(x) v_2(y)$, then the result follows from the duality arguments.
We give the proof, for example,  for $H^{a,b}$ in the case when $w(x,y)= w_1(x) w_2(y)$.

First suppose that $S:=\int\limits_{G_2}w^{1-p'}_2(y)dy =\infty$. Let $\{x_k\}_{k=-\infty}^{+\infty}$ be a
sequence of positive numbers for which the equality

\begin{equation}\label{1.1.16}
2^k=\int\limits_{B(e_2, b x_k)} w_2^{1-p'}(y) dy
\end{equation}
holds for all $k\in {\mathbb{Z}}$. This equality  follows because of the continuity in $t$ of the integral over the ball $B(e_2, bt)$.  It is clear that $\{x_k\}$ is increasing
and ${\mathbb{R}}_+ = \cup_{k\in {\mathbb{Z}}}[x_k, x_{k+1})$. Moreover, it is easy to verify that
$$ 2^k = \int\limits_{B(e_2, b x_{k+1})\setminus B(e_2, b x_{k})} w_2^{1-p'}(y)dy. $$  Let $f\geq 0$. We have that
\begin{eqnarray*} && \|H^{a,b}f\|^q_{L^q_{v}(G_1\times G_2)}= \iint\limits_{G_1\times G_2}  v(x,y) \big(H^{a,b} f\big)^q (x,y) dxdy \\
 &\leq& \sum_{k \in {\mathbb{Z}}} \int\limits_{G_1} \int\limits_{B(e_2, x_{k+1})\setminus B(e_2,  x_{k})} v(x,y) \bigg( \iint\limits_{ B\big(e_1, a r_1(x)\big)\times B\big(e_2, b r_2(x)\big)}
f(t,\tau)dt d\tau \bigg)^q dxdy
\\ &\leq& \sum_{k \in {\mathbb{Z}}} \int\limits_{G_1}  \bigg(\int\limits_{B(e_2, x_{k+1})\setminus B(e_2,  x_{k})} v(x,y) dy\bigg)\bigg( \int\limits_{B(e_1, ar_1(x))} \bigg(\int\limits_{B(e_2, bx_{k+1})}
f(t,\tau)d\tau\bigg)dt \bigg)^q dx \\ &=& \sum_{k\in {\mathbb{Z}}} \int\limits_{G_1}  V_k(x)\bigg( \int\limits_{B(e_1, ar_1(x))} F_k(t) dt \bigg)^q dx,
\end{eqnarray*}
where
$$V_k(x):=\int\limits_{B(e_2,  x_{k+1})\setminus B(e_2,  x_{k})} v(x,y) dy;\;\; F_k(t):= \int\limits_{B(e_2,  b x_{k+1})} f(t,\tau)d\tau.$$

It is obvious that
$$ A^q \geq \sup_{\substack{a>0 \\ j\in {\mathbb{Z}}}} \bigg(\int\limits_{G_1\setminus B(e_1, t)}v_j (y) dy\bigg) \bigg(\iint\limits_{B(e_1, at)\times B(e_2, b x_j)} w^{1-p'}(x,y)dxdy\bigg)^{q/p'}. $$
Hence, by Theorem A
$$ \|H^{a,b}f\|^q_{L^q_{v}(G_1 \times G_2)} \leq  c A^q \sum_{j\in {\mathbb{Z}}}\bigg[ \int\limits_{G_1} w_1(x) \bigg(\int\limits_{B(e_2, bx_j)} w_{2}^{1-p'}(y)dy\bigg)^{1-p} (F_k(x))^p dx \bigg]^{q/p} $$
$$ \leq cA^q \bigg[ \int\limits_{G_1} w_1(x) \sum_{j\in Z} \bigg( \int\limits_{B(e_2, b x_j)} w_{2}^{1-p'}(y)dy\bigg)^{1-p} \\  \bigg(\sum_{k=-\infty}^{j} \int\limits_{B(e_2, b x_{k+1}) \setminus B(e_2, b x_{k})}
f(x,\tau)d\tau\bigg)^p dx \bigg]^{q/p}.
$$
 On the other hand,
\eqref{1.1.16} yields that
\begin{eqnarray*} && \sum_{k=n}^{+\infty} \bigg( \int\limits_{B(e_2, b x_k)} w_2^{1-p'}(y) dy \bigg)^{1-p} \bigg(\sum_{k=-\infty}^{n} \int\limits_{B(e_2, b x_{k+1})\setminus B(e_2, b x_{k})}w_2^{1-p'}(y) dy \bigg)^{p-1}\\ & =& \sum_{k=n}^{+\infty} \bigg( \int\limits_{B(e_2, b x_k)} w_2^{1-p'}(y) dy \bigg)^{1-p} \bigg( \int\limits_{B(e_2, b x_{n+1})}w_2^{1-p'}(y) dy \bigg)^{p-1}\!\!\!=
\Big(\sum_{k=n}^{+\infty}2^{k(1-p)}\Big)2^{(n+1)(p-1)}\leq c
\end{eqnarray*}
for all $n\in {\mathbb{Z}}$. Hence by the discrete Hardy inequality (see e.g. \cite{Ben}) and H\"{o}lder's
inequality we have
\begin{eqnarray*} \|H^{a,b}f\|^q_{L^q_{v}(G_1 \times G_2)} &\leq& c A^q \bigg[ \int\limits_{G_1} w_1(x) \sum_{j\in {\mathbb{Z}}} \bigg( \int\limits_{B(e_2, b x_{j+1})\setminus B(e_2, b x_{j})}  w_{2}^{1-p'}(y)dy\bigg)^{1-p} \bigg(\int\limits_{B(e_2, b x_{j+1})\setminus B(e_2, bx_{j})} f(x,\tau)d\tau\bigg)^p dx \bigg]^{q/p} \\ &\leq& c A^q \bigg[ \int\limits_{G_1}  w_1(x) \sum_{j\in {\mathbb{Z}}}  \bigg( \int\limits_{B(e_2, b x_{j+1})\setminus B(e_2, b x_{j})}
w_2(\tau)f^p(x,\tau)d\tau\bigg) dx \bigg]^{q/p}=  c A^q \|f\|_{L^p_w(G_1 \times G_2)}^q.
\end{eqnarray*}

If  $S<\infty$, then without loss of generality we can assume that
$S=1$. In this case we  choose the sequence $\{ x_k \}_{k=-\infty}^{0}$ for which (\ref{1.1.16}) holds for all $k\in {\mathbb{Z}}_-$.
Arguing as in the case $S=\infty$ and using  slight modification of the discrete Hardy inequality (see also \cite{KMP}, Chapter 1 for similar arguments), we finally obtain the desired result.

Finally we notice that the  part (i) can be also proved if we first  establish the  boundedness of the operator $({\mathcal{H}}^{a,b}\varphi)(t,\tau) =\int\limits_0^{at}\int\limits_0^{b\tau} \varphi(s,r) dsdr $ in the spirit of Theorem 1.1.6 in \cite{KMP} and then pass to the case of $G_1 \times G_2$ by Proposition \ref{polar}.
\end{proof}

\vskip+0.2cm

The next statement will be useful for us.

\begin{prop}\label{double-Hardy}  Let $1<p\leq q<\infty$. Assume that $v$ and $w$ are weights on $G_1\times G_2$. Suppose that $w(x,y)=  w_1(x)w_2(y)$ and that $W_i(\infty)=\infty$, $i=1,2$. Then  the operator $H^{1,1}$ is bounded from $L^p_{dec,r}(w, G_1\times G_2)$ to $L^q(v, G_1 \times G_2)$ if and only if the following four conditions are satisfied:

${\rm{(i)}}$
$$ \sup_{a_1, a_2>0} \bigg(\int\limits_{B(e_1, a_1)}  \int\limits_{B(e_2, a_2)}  w(x,  y) dx dy \bigg)^{-1/p} \bigg(\int\limits_{B(e_1, a_1)}  \int\limits_{B(e_2, a_2)}  r_1^{Q_1 q}(x) r_2(y) ^{Q_2 q} v(x,y) dx dy\bigg)^{1/q}< \infty;$$

${\rm{(ii)}}$
$$\sup_{a_1, a_2>0} \bigg( \int\limits_{B(e_1, a_1)}  \int\limits_{B(e_2, a_2)}   r_1^{Q_1 p'} (x) r_2(y)^{Q_2 p'} W^{-p'}(r_1(x), r_2(y)) w(x,y)
dx dy  \bigg)^{1/p'} \bigg(\int\limits_{G_1\setminus B(e_1, a_1)}  \int\limits_{G_2 \setminus B(e_2, a_2)}  v(x,y)dx dy \bigg)^{1/q} <\infty; $$

${\rm{(iii)}}$
$$\sup_{a_1, a_2>0} \bigg( \int\limits_{B(e_1, a_1)} w_1(r_1(x)) dx\bigg)^{-1/p} \bigg(\int\limits_{B(e_2, a_2)} r_2(y)^{Q_2p'} W_2^{-p'}(r_2(y))w_2(y)dy\bigg)^{1/p'} \bigg(\int\limits_{B(e_1, a_1)} \int\limits_{G_2\setminus B(e_2, a_2)} r_1(x)^{Q_1 q} v(x,y)dx dy \bigg)^{1/q} <\infty; $$

${\rm{(iv)}}$
$$\sup_{a_1, a_2>0} \bigg(\int\limits_{B(e_1, a_1)} r_1(x)^{Q_1p'}W_1^{-p'}(r_1(x))w_1(x)dt_1\bigg)^{1/p'}\bigg( \int\limits_{B(e_2, a_2)} w_2(y)d y\bigg)^{-1/p}  \bigg(\int\limits_{G_1\setminus B(e_1, a_1)} \int\limits_{B(e_2, a_2)}   r_2(y)^{Q_2q} v(x,y)dx dy \bigg)^{1/q} <\infty. $$
\end{prop}

\begin{proof} We follow the proof of Theorem 5.3 in \cite{BHP}. First of all observe that by Proposition \ref{BHP2},  if $w$ is a product
weight, i.e., $w(x_1,  x_2)= w_1(x_1) w_2(x_2)$, such that $W_i(\infty)=\infty$, $i=1, 2$, and $v$ is any weight on
$G_1\times G_2$, then $H^{1,1}$ is bounded from $L^p_{dec,r}(w, G_1)$ to $L^q(v, G_2)$ if and only if

\begin{equation*}\bigg( \iint\limits_{G_1 \times G_2} \bigg( \int\limits_{B(e_1, r_1(x))}    \int\limits_{B(e_2, r_2(x))} \bigg[ \int\limits_{G_1 \setminus B(e_1, r_1(t))}  \int\limits_{G_1 \setminus B(e_2, r_2(\tau))} g(s,  \varepsilon) ds d\varepsilon\bigg] dt d\tau\bigg)^{p'} W^{-p'}(r_1(x), r_2(y)) w(x,y)dx dy \bigg)^{1/p'}
\end{equation*}
\begin{equation}\label{dual}\leq
c \bigg( \iint\limits_{G_1\times G_2} g^{q'}(x,y) v^{1-q'}(x,y) dx dy  \bigg)^{1/q'}, \;\; g\geq 0.
\end{equation}
Further, we have that

\begin{eqnarray*} && \iint\limits_{B(e_1, r_1(x))\times B(e_2, r_2(x))}  \bigg( \int\limits_{G_1 \setminus B(e_1, r_1(t))}  \int\limits_{G_2 \setminus B(e_2, r_2(t))} g(s,\varepsilon) ds d\varepsilon\bigg) dt d\tau \\
 &=& \int\limits_{B(e_1, r_1(x))}    \int\limits_{B(e_2, r_2(x))}  r_1^{Q_1}(t) r_2^{Q_2}(\tau) g(t,\tau) dt d\tau  +  r_1^{Q_1}(x) \int\limits_{G_1 \setminus B(e_1, r_1(x))} \int\limits_{B(e_2, r_2(y))}  r_2^{Q_2}(\tau) g(t, \tau) dt d\tau \\ & +& r_2^{Q_2}(y) \int\limits_{B(e_1, r_1(x))} \int\limits_{G_2 \setminus B(e_2, r_2(y))}  r_1^{Q_1}(t)  g(t, \tau) dt d\tau \\&+&  r_1^{Q_1}(x)r_2^{Q_2}(y) \int\limits_{G_1 \setminus B(e_1, r_1(x))}\int\limits_{G_2 \setminus B(e_2, r_2(y))} g(t, \tau) dt d\tau \\ &=:& I^{(1)}(x,y) + I^{(2)}(x,y)+ I^{(3)}(x,y)+ I^{(4)}(x,y).
 \end{eqnarray*}

It is obvious that \eqref{dual}  holds  if and only if
\begin{equation}\label{j-inequality} \bigg( \iint\limits_{G_1\times G_2} (I^{(j)})^{p'}(x,y) W^{-p'}(r_1(x), r_2(y))w(x,y) dxdy\bigg)^{1/p'}
\leq c \bigg( \iint\limits_{G_1\times G_2} g^{q'}(x,y)  v^{1-q'}(x,y)dx dy \bigg)^{1/q'}
\end{equation}
for $j= 1, 2, 3,4. $ By using  Proposition \ref{Hardy-G_1-G_2} (Part (i)) we find that
$$\Bigg( \iint\limits_{G_1 \times G_2} (I^{(1)})^{p'}(x,y) W^{-p'}(r_1(x), r_2(y))w(x,y) dxdy \Bigg)^{1/p'}\leq c  \Bigg( \iint\limits_{G_1 \times G_2}  g^{q'}(x,y) v^{1-q'}(x,y) dx dy \Bigg)^{1/q'} $$
if and only if

 \begin{eqnarray*} && \bigg( \int\limits_{G_1 \setminus B(e_1, t)} \int\limits_{G_2 \setminus B(e_2, \tau)}  W^{-p'}(r_1(x), r_2(y)) w(x,y)dx dy  \bigg)^{1/p'} \bigg( \iint\limits_{B(e_1, t)\times B(e_2,\tau)} \bigg( \frac{v^{1-q'}(x,y)}{r_1^{Q_1 q'}(x) r_2^{Q_2 q'}(y)}\bigg)^{1-q} dxdy \bigg)^{1/q} \\ &=& c_p \bigg( \iint\limits_{B(e_1, t)\times B(e_2, \tau)} w(x,y) dx dy \bigg)^{-1/p} \bigg( \iint\limits_{B(e_1, t)\times B(e_2, \tau)}  v(x,y) r_1^{Q_1 q}(x)r_2^{Q_2 q}(y) dx dy \bigg)^{1/q} \leq C.
 \end{eqnarray*}
In the latter equality we used the equality

$$  \bigg( \int\limits_{G_i \setminus B(e_i, t)}  W_i^{-p'}\big(r_i(x)\big) w_i(x) dx  \bigg)^{1/p'}=\bigg( \int\limits_{B(e_i, t)} w_i(x) dx \bigg)^{-1/p}, \;\;\ i=1,2,$$
which is direct consequence of integration by parts and Proposition \ref{polar}. Taking now Proposition  \ref{Hardy-G_1-G_2} (Part (ii)) into account  we find that \eqref{j-inequality} holds for $j=4$ if and only if condition (ii) is satisfied, while
Proposition \ref{Hardy-G_1-G_2} (Parts (iii) and (iv)) and the following observation:

\begin{eqnarray*}&& \sup_{a_1, a_2>0} \bigg( \int\limits_{G_1 \setminus B(e_1, a_1)} w_1(x) W_1^{-p'}(r_1(x))dx\bigg)^{1/p'}
\bigg(\int\limits_{B(e_2, a_2)} r_2^{p' Q_2}(y) W_2^{-p'}(r_2(y))w_2(y)dy\bigg)^{1/p'} \\ &\times& \bigg(\int\limits_{B(e_1, a_1)}  \int\limits_{G_2 \setminus B(e_2, a_2)} r_1^{Q_1q}(x) v(x,y)dx dy \bigg)^{1/q}\\ &=& c_p \sup_{a_1, a_2>0} \bigg( \int\limits_{B(e_1, a_1)} w_1(x) dx \bigg)^{-1/p} \bigg(\int\limits_{B(e_2, a_2)} r_2^{Q_2p'}(y)  W_2^{-p'}(r_2(y)) w_2(y)dy\bigg)^{1/p'}
\\&\times& \bigg(\int\limits_{B(e_1, a_1)}  \int\limits_{G_2\setminus B(e_2, a_2)} r_1^{Qq}(x)  v(x,y)dx dy \bigg)^{1/q} <\infty;
\end{eqnarray*}

\begin{eqnarray*} && \sup_{a_1, a_2>0} \bigg(\int\limits_{B(e_1, a_1)}  r_1^{Q_1p'}(x) W_1^{-p'}(r_1(x))w_1(x)dx \bigg)^{1/p'}\bigg( \int\limits_{G_2 \setminus B(e_2, a_2)} w_2(y)W_2^{-p'}(r_2(y)) dy \bigg)^{1/p'}
\\
&\times & \bigg(\int\limits_{G_1 \setminus B(e_1, a_1)} \int\limits_{B(e_2, a_2)}   r_2^{Q_2 q}(y)  v(x,y)dx dy  \bigg)^{1/q}  \\ &=& c_p\sup_{a_1, a_2>0} \bigg(\int\limits_{B(e_1, a_1)} r_1^{Q_1 p'}(x) W_1^{-p'}(r_1(x))w_1(x)dx \bigg)^{1/p'}
\bigg( \int\limits_{B(e_2, a_2)} w_2(t_2) dt_2\bigg)^{-1/p}   \\
&\times & \bigg(\int\limits_{G_1\setminus B(e_1, a_1)}\int\limits_{B(e_2, a_2)}    r_2^{Q_2 q}(y)  v(x,y)dx dy  \bigg)^{1/q} <\infty
\end{eqnarray*}
yield \eqref{j-inequality} for $j=2,3$.
\end{proof}

Let
$$(J_{\alpha_1,\alpha_2}f)(x,y)= \int\limits_{B \big(e_1, 2 c_0 r_1(x)\big)}\int\limits_{B \big(e_2, 2 c_0 r_2(y)\big)} f(t,\tau) r_1(xt^{-1})^{\alpha_1-Q_1} r_2(y\tau^{-1})^{\alpha_2-Q_2} dt d\tau, $$

$$(J_{\alpha_1}S_{\alpha_2}f)(x,y)= \int\limits_{B \big(e_1, 2 c_0 r_1(x)\big)}\int\limits_{G_2 \setminus B\big(e_2, 2 c_0 r_2(y)\big)} f(t,\tau) r_1(xt^{-1})^{\alpha_1-Q_1} r_2(y\tau^{-1})^{\alpha_2-Q_2} dt d\tau, $$

$$(S_{\alpha_1}J_{\alpha_2}f)(x,y)= \int\limits_{G_1 \setminus B \big(e_1, 2 c_0 r_1(x)\big)}\int\limits_{ B\big(e_2, 2 c_0 r_2(y)\big)} f(t,\tau) r_1(xt^{-1})^{\alpha_1-Q_1} r_2(y\tau^{-1})^{\alpha_2-Q_2} dt d\tau, $$

$$(S_{\alpha_1, \alpha_2}f) (x,y) = \int\limits_{G_1 \setminus B \big(e_1, 2 c_0 r_1(x)\big)}\int\limits_{G_2\setminus B \big(e_2, 2 c_0 r_2(y)\big)} f(t,\tau) r_1(xt^{-1})^{\alpha_1-Q_1} r_2(y\tau^{-1})^{\alpha_2-Q_2} dt d\tau, $$
where $c_0$ is the constant from the triangle inequality for the homogeneous norms $r_1$ and $r_2$.

It is obvious that
\begin{equation}\label{rep}
I_{\alpha_1, \alpha_2}f = J_{\alpha_1,\alpha_2}f+ J_{\alpha_1}S_{\alpha_2}f +S_{\alpha_1}J_{\alpha_2}f +S_{\alpha_1, \alpha_2}f.
\end{equation}

Now we formulate the main result of this section.
\begin{thm}\label{Main-theorem}   Let $1<p\leq q<\infty$. Assume that $v$ and $w$ are weights on $G_1\times G_2$ such that $w(x,y)=  w_1(x)w_2(y)$. Suppose that either $w_i\in DC^{\alpha_i,p}$, $i=1,2$, or $v\in DC(x) \cap DC(y)$.  Then  the operator $I_{\alpha_1, \alpha_2}$ is bounded from $L^p_{dec,r}(w, G_1\times G_2)$ to $L^q(v, G_1 \times G_2)$ if and only if the following  conditions are satisfied:

{\rm{(i)}}

$$ A_1:= \sup_{a_1, a_2>0} \bigg(\int\limits_{B(e_1, a_1)}  \int\limits_{B(e_2, a_2)} w(x,y) dx dy\bigg)^{-1/p}  \bigg(\int\limits_{B(e_1, a_1)}  \int\limits_{B(e_2, a_2)} \Big( r_1^{\alpha_1}(x)  r_2^{\alpha_2}(y) \Big)^qv(x,y) dx dy \bigg)^{1/q}< \infty; $$

{\rm{(ii)}}
$$A_2:= \sup_{a_1, a_2>0} \bigg(\int\limits_{B(e_1, a_1)}  \int\limits_{B(e_2, a_2)}  r_1^{Q_1 p'}(x)r_2^{Q_2 p'}(y)W^{-p'}(r_1(x), r_2(y)) w(x,y)
dx dy  \bigg)^{1/p'} $$
$$ \times \bigg(\int\limits_{G_1 \setminus B(e_1, a_1)}  \int\limits_{G_2 \setminus B(e_2, a_2)} \Big( r_1^{\alpha_1 -Q_1}(x) r_2^{\alpha_2 -Q_2}(y)\Big)^qv(x,y )dx dy \bigg)^{1/q} <\infty;$$

{\rm{(iii)}}
$$A_3:= \sup_{a_1, a_2>0} \bigg( \int\limits_{B(e_1, a_1)} w_1(x) dx\bigg)^{-1/p} \bigg(\int\limits_{B(e_2, a_2)} r_2^{Q_2 p'}(y) W_2^{-p'}(r_2(y))w_2(y)dy\bigg)^{1/p'} $$
$$\times \bigg(\int\limits_{B(e_1, a_1)}   \int\limits_{G_2\setminus B(e_2, a_2)} r_1^{\alpha_1 q}(x) r_2^{  q(\alpha_2 -Q_2)}(y) v(x,y)dx dy \bigg)^{1/q} <\infty; $$
\rm{(iv)}
$$A_4:= \sup_{a_1, a_2>0} \bigg(\int\limits_{B(e_1, a_1)}   r_1^{ Q_1 p'} (x) W_1^{-p'}(r_1(x))w_1(x)dx \bigg)^{1/p'}\bigg( \int\limits_{B(e_2, a_2)} w_2(y)dy \bigg)^{-1/p}  $$
$$ \bigg(\int\limits_{G_1\setminus B(e_1, a_1)}\int\limits_{B(e_2, a_2)}   r_1^{q(\alpha_1 -Q_1)}(x) r_2^{q\alpha_2} (y) v(x,y)dx dy \bigg)^{1/q} <\infty. $$

\rm{(v)}
$$ A_5:= \sup_{a_1, a_2>0} \bigg(\int\limits_{ G_1 \setminus B(e_1, a_1)}  \int\limits_{G_2\setminus B(e_2, a_2)}  r_1^{\alpha_1 p'}(x) r_2^{\alpha_2 p'}(y) W^{-p'}(r_1(x), r_2(y)) w(x,y) dx dy\bigg)^{1/p'}  $$
$$ \times \bigg(\int\limits_{B(e_1, a_1)}  \int\limits_{B(e_2, a_2)}  v(x,y) dx dy \bigg)^{1/q}< \infty; $$

\rm{(vi)}
$$ A_6:=\sup_{a_1, a_2>0} \bigg(\int\limits_{ B(e_1, a_1)}  w_1(x) dx\bigg)^{-1/p} \bigg(\int\limits_{G_2\setminus B(e_2, a_2)}  r_2^{\alpha_2 p'}(y) W_2^{-p'}(r_2(y)) w_2(y) dy\bigg)^{1/p'}  $$
$$\times \bigg(\int\limits_{B(e_1, a_1)}  \int\limits_{B(e_2, a_2)} r_1^{\alpha_1 q}(x) v(x,y) dx dy \bigg)^{1/q}< \infty; $$

\rm{(vii)}
$$ A_7:= \sup_{a_1, a_2>0} \bigg(\int\limits_{ B(e_1, a_1)}  r_1^{Q_1 p'}(x)  W_1^{-p'}(r_1(x)) w_1(x) dx\bigg)^{1/p'} \bigg( \int\limits_{G_2\setminus B(e_2, a_2)}  r_2^{\alpha_2 p'}(y) W_2^{-p'}(r_2(y)) w_2(y)  dy\bigg)^{1/p'}$$
  $$\times \bigg(\int\limits_{G_1\setminus B(e_1, a_1)}  \int\limits_{B(e_2, a_2)} r_1^{(\alpha_1 -Q_1)q}(x) v(x,y) dx dy \bigg)^{1/q}< \infty; $$

\rm{(viii)}
$$ A_8:= \sup_{a_1, a_2>0} \bigg(\int\limits_{G_2\setminus B(e_1, a_1)} r_1^{\alpha_1 p'}(x) W_1^{-p'}(r_1(x)) w_1(x) dx\bigg)^{-1/p} \bigg( \int\limits_{B(e_2, a_2)} w_2(y)  dy\bigg)^{1/p'}   $$
$$\times \bigg(\int\limits_{B(e_1, a_1)}  \int\limits_{B(e_2, a_2)} r_2^{\alpha_2 q}(x) v(x,y) dx dy \bigg)^{1/q}< \infty; $$

\rm{(ix)}
$$ A_9:= \sup_{a_1, a_2>0} \bigg(\int\limits_{B(e_1, a_1) }  r_2^{Q_2 p'}(y)  W_2^{-p'}(r_2(y)) w_2(y) dy\bigg)^{1/p'} \bigg(\int\limits_{G_1 \setminus B(e_1, a_1) }  r_1^{\alpha_1 p'}(x) W_1^{-p'}(r_1(x)) w_1(x)  dx\bigg)^{1/p'} $$
 $$\times \bigg(\int\limits_{B(e_1, a_1)}  \int\limits_{G_2 \setminus B(e_2, a_2)} r_2^{(\alpha_2 -Q_2)q}(y) v(x,y) dx dy \bigg)^{1/q}< \infty. $$

\end{thm}

\begin{proof}

Let us assume that $v\in DC(x) \cap DC(y)$. The case when $w_i\in DC^{\alpha_i,p}(G_i)$, $i=1,2$  follows analogously. By using representation \eqref{rep} we have to investigate the boundedness of the operators  $J_{\alpha_1,\alpha_2}f$,  $J_{\alpha_1}S_{\alpha_2}f$, $S_{\alpha_1}J_{\alpha_2}f$, $S_{\alpha_1, \alpha_2}f$ separately.

Since $f\in \mathcal{DR}(G_1\times G_2)$ by using the arguments of the proof of Proposition \ref{main1} it can be checked that
$$  (J_{\alpha_1,\alpha_2}f)(x,y)  \approx r_1^{\alpha_1-Q_1}(x)  r_2^{\alpha_2-Q_2}(y)  \iint\limits_{B(e_1, r_1(x))\times B(e_2, r_2(y))} f(t,\tau) dt d\tau $$
(see also \cite {MMS} for similar estimate in the case of the multiple one-sided potentials on ${\Bbb{R}}_+^2$).
Hence, by Proposition \ref{double-Hardy} we have that  $J_{\alpha_1,\alpha_2}$ is bounded from $L^p_{dec,r}(w, G_1\times G_2)$ to $L^q(v, G_1 \times G_2)$ if and only if conditions (i)- (iv) hold.

Observe that the dual to $S_{\alpha_1, \alpha_2}$ is given by
$$(S^{*}_{\alpha_1, \alpha_2} g)(x,y)=  \iint\limits_{B(e_1, r_1(x)/(2c_0))\times B(e_2, r_2(y)/(2c_0))}g(t,\tau) r_1^{\alpha_1-Q_1}(xt^{-1}) r_2^{\alpha_2-Q_2}(y\tau^{-1})dt d\tau. $$
Further, Tonelli's theorem together with  Lemma \ref{estimate} for both variables implies that there are positive constants $c_1$ and $c_2$ such that for all $(x,y)\in G_1 \times G_2$ for the dual  (see also the proof of Lemma  \ref{main2})

\begin{eqnarray*} &&
r_1^{\alpha_1}(x) r_2^{\alpha_2}(y) \iint\limits_{B(e_1, r_1(x)/(4c_0))\times B(e_2, r_2(y)/(4c_0))}g(t,\tau) dt d\tau \leq c_1 \iint\limits_{B(e_1, r_1(x))\times B(e_2, r_2(y))} \big( S^*_{\alpha_1, \alpha_2} g\big) (t, \tau) dt d\tau \\
&\leq& c_2
r_1^{\alpha_1}(x) r_2^{\alpha_2}(y) \iint\limits_{B(e_1, r_1(x)/(2c_0))\times B(e_2, r_2(y)/(2c_0))}g(t,\tau) dt d\tau.
\end{eqnarray*}
Applying Propositions \ref{Hardy-G_1-G_2} and  \ref{double-Hardy} with the condition that $v\in DC(G_1\times G_2)$ we find that the operator  $S_{\alpha_1, \alpha_2}$ is bounded from $L^p_{dec,r}(w, G_1\times G_2)$ to $L^q(v, G_1 \times G_2)$ if and only if condition (v) is satisfied.

Further, observe that due to the fact that $f$ is radially decreasing with respect to the first variable we have

$$ (J_{\alpha_1} S_{\alpha_2} f)(x,y)\approx ({\mathcal{H}}_{\alpha_1}  S_{\alpha_2} f)(x,y),$$
where
$$({\mathcal{H}}_{\alpha_1}  S_{\alpha_2} f)(x,y) = r_1^{\alpha_1-Q_1}(x) \int\limits_{B\big(e_1, 2 c_0 r_1(x)\big)}\int\limits_{ G_2 \setminus B\big(e_2, 2 c_0 r_2(y)\big)} f(t,\tau)  r_2(y\tau^{-1})^{\alpha_2-Q_2} dt d\tau. $$

Dual of ${\mathcal{H}}_{\alpha_1}  S_{\alpha_2}$ is given by
$$ \big( {\mathcal{H}}^{*}_{\alpha_1}  S^{*}_{\alpha_2} g\big)(t,\tau) =   \int\limits_{G_1\setminus B(e_1, r(t))}
\int\limits_{B(e_2, r(\tau)/2c_0)}r_1^{\alpha_1-Q_1}(s) r_2^{\alpha_2-Q_2}(\varepsilon \tau ^{-1}) f(s,\varepsilon) ds d\varepsilon. $$
Further, we have
 \begin{eqnarray*}
 T(x,y)&:=&  \iint\limits_{B(e_1, r_1(x))\times B(e_2, r_2(y))} ({\mathcal{H}}^{*}_{\alpha_1}  S^{*}_{\alpha_2} g)(t,\tau) dt d\tau \\
 &=& \iint\limits_{B(e_1, r_1(x))\times B(e_2, r_2(y))} \bigg( \int\limits_{B(e_1, r_1(x))\setminus B(e_1, r(t))}
\int\limits_{B(e_2, r(\tau)/2c_0)}r_1^{\alpha_1-Q_1}(s) r_2^{\alpha_2-Q_2}(\tau \varepsilon^{-1}) f(s,\varepsilon) ds d\varepsilon\bigg) dt d\tau \\
&+&  \iint\limits_{B(e_1, r_1(x))\times B(e_2, r_2(y))} \bigg( \int\limits_{G_1\setminus B(e_1, r_1(x))}
\int\limits_{B\big(e_2, r(\tau)/(2c_0)\big)}r_1^{\alpha_1-Q_1}(s) r_2^{\alpha_2-Q_2}(\tau \varepsilon^{-1}) f(s,\varepsilon) ds
d\varepsilon\bigg) dt d\tau \\
&=:&  T_1(x,y)+T_2(x,y).
\end{eqnarray*}
 Tonelli's theorem for $G_1$ , the inequality $r_2^{\alpha_2-Q_2}(\tau \varepsilon^{-1})\geq c r_2^{\alpha_2-Q_2}(y)$ for $\tau\in B(e_2, r(y))$, $\varepsilon \in B\big(e_2, r(\tau)/(2c_0)\big)$, and the fact that the integral $\int\limits_{B(e_1, \tau)} f(s, \varepsilon) ds $ is decreasing in  $\tau$ uniformly  to $\varepsilon$ yield that
\begin{eqnarray*} T_1(x,y) &\geq& c r_2^{\alpha_2-Q_2}(y) \int\limits_{B(e_1, r_1(x))} \int\limits_{B(e_2, r_2(y))\setminus B(e_2, r_2(y)/2)} \bigg( \int\limits_{B(e_1, r_1(x))\setminus B(e_1, r(t))}
\int\limits_{B\big(e_2, r_2(y)/(4c_0)\big)}r_1^{\alpha_1-Q_1}(s)  f(s,\varepsilon) ds d\varepsilon\bigg) dt d\tau \\
&= & c  r_2^{\alpha_2}(y) \int\limits_{B(e_1, r_1(x))}  \bigg( \int\limits_{B(e_1, r_1(x))\setminus B(e_1, r(t))}\bigg( r_1^{\alpha_1-Q_1}(s) \bigg( \int\limits_{B\big(e_2, r_2(y)/(4c_0)\big)}   f(s,\varepsilon)  d\varepsilon\bigg)  ds \bigg) dt \\
&=& c  r_2^{\alpha_2}(y) \int\limits_{B(e_1, r_1(x))}  \bigg( \int\limits_{B(e_1, r_1(x))\setminus B(e_1, r(t))} F(s,y)  ds \bigg) dt = c  r_2^{\alpha_2}(y) \int\limits_{B(e_1, r_1(x))}  F(s,y) \bigg( \int\limits_{ B(e_1, r(s))} dt \bigg)ds \\ &=& c  r_2^{\alpha_2}(y) \int\limits_{B(e_1, r_1(x))}    \int\limits_{B\big(e_2, r_2(y)/(4c_0)\big)}   r_1^{\alpha_1}(s) f(t,\tau)  d\varepsilon  ds .
\end{eqnarray*}
Here we used the notation
$$F(s,y): = \int\limits_{B\big(e_2, r_2(y)/(4c_0)\big)}   f(s,\varepsilon)  d\varepsilon.$$

Taking into account that the function $\int\limits_{B(e_2, 2c_0\lambda)} f(s, \varepsilon) d\varepsilon $ is decreasing in $\lambda$ uniformly to $s$, the inequality  $r_2(\tau\varepsilon^{-1})\leq c r_2(y)$ for $\tau\in B(e_2, r(y))$, $\varepsilon \in B\big(e_2, r(\tau)/(2c_0)\big)$, and Tonelli's theorem for $G_1$ we find that

$$  T_2(x,y) \geq c  r_1^{Q_1}(x)  r_2^{\alpha_2}(y) \int\limits_{G_1\setminus B(e_1, r_1(x))}    \int\limits_{B\big(e_2, r_2(y)/(4c_0)\big)}   r_1^{\alpha_1-Q_1}(s) f(t,\tau)  d\varepsilon  ds . $$

To get the upper estimate, observe that  Tonelli's theorem for $G_1\times G_2$ and Lemma \ref{estimate} for $r_2$ yield that
\begin{eqnarray*}
T_1(x,y) &\leq&  \int\limits_{B(e_1, r_1(x))} \int\limits_{B\big(e_2, r_2(y)/(2c_0)\big)}  r_1^{\alpha_1-Q_1}(s)f(s,\varepsilon)  \bigg( \int\limits_{B(e_1, r_1(s))}
\int\limits_{B\big(e_2, r_2(y)\big)\setminus B\big(e_2, 2c_0 r_2(\varepsilon)\big)}r_2^{\alpha_2-Q_2}(\tau\varepsilon^{-1})   dt d\tau \bigg) ds  d\varepsilon \\
&\leq& c r_2^{\alpha_2}(y)  \iint\limits_{B\big(e_1, r_1(x)\big)\times B\big(e_2, r_2(y)/(2c_0)\big)}  r_1^{\alpha_1}(s) f(s,\varepsilon) ds d\varepsilon.
\end{eqnarray*}

Similarly,
$$ T_2(x,y) \leq  c r_1^{Q_1}(x) r_2^{\alpha_2}(y) \iint\limits_{G_1 \setminus B\big(e_1, r_1(x)\big)\times B\big(e_2, r_2(y)/(2c_0)\big)}  r_1^{\alpha_1-Q_1}(s) f(s,\varepsilon)  ds d\varepsilon. $$

Summarazing these estimates we see that there are positive constants $c_1$ and $c_2$ depending only on $\alpha_1$, $\alpha_2$, $Q_1$ and $Q_2$ such that

\begin{eqnarray*}
&& r_2^{\alpha_2}(y)   \iint\limits_{B\big(e_1, r_1(x)\big)\times B\big(e_2, r_2(y)/(4c_0)\big)}  r_1^{\alpha_1}(s) f(s,\varepsilon) ds d\varepsilon \\
&+& r_1^{Q_1}(x) r_2^{\alpha_2}(y)   \iint\limits_{G_1 \setminus B\big(e_1, r_1(x)\big)\times B\big(e_2, r_1(y)/(4c_0)\big)}  r_1^{\alpha_1-Q_1}(s) f(s,\varepsilon)  ds d\varepsilon.
 \\
 &\leq & c_1 T(x,y) \leq r_2^{\alpha_2}(y)   \iint\limits_{B\big(e_1, r_1(x)\big)\times B\big(e_2, r2(y)/(2c_0)\big)}  r_1^{\alpha_1}(s) f(s,\varepsilon)  ds d\varepsilon \\
 &+& r_1^{Q_1}(x) r_2^{\alpha_2}(y)   \iint\limits_{G_1 \setminus B\big(e_1, r_1(x)\big)\times B\big(e_2, r_1(y)/(2c_0)\big)}  r_1^{\alpha_1-Q_1}(s) f(s,\varepsilon)  ds d\varepsilon.
\end{eqnarray*}

Taking Propositions \ref{Hardy-G_1-G_2} and  \ref{BHP2} into account together with the doubling condition for $v$ with respect to the second variable we see that the operator $J_{\alpha_1} S_{\alpha_2}$ is bounded from $L^p_{dec,r}(w,G_1)$ to $L^q(v,G_2)$ if and only if the conditions (vi) and (vii) are satisfied.

By the similar manner (changing the roles of the first and second variables) we can get that $S_{\alpha_1} J_{\alpha_2}$ is bounded from $L^p_{dec,r}(w,G_1)$ to $L^q(v,G_2)$ if and only if the conditions (viii) and (ix) are satisfied.

 \end{proof}

 Theorem \ref{Main-theorem} and Remark \ref{Rem}   imply  criteria for the trace inequality for $I_{\alpha_1, \alpha_2}$. Namely the following statement holds:

\begin{thm}\label{Main-theorem-1}   Let $1<p\leq q<\infty$ and let $0<\alpha_i<Q_i/p$, $i=1,2$. Then $I_{\alpha_1, \alpha_2}$ is bounded from $L^p_{dec,r}(G_1\times G_2)$ to $L^q(v, G_1 \times G_2)$ if and only if the following  condition holds

$$ B:= \sup_{a_1, a_2>0} \bigg(\int\limits_{B(e_1, a_1)}  \int\limits_{B(e_2, a_2)} v(x,y) dx dy\bigg)^{1/q}  a_1^{\alpha_1-Q_1/p} a_2^{\alpha_2-Q_2/p} < \infty. $$

\end{thm}

\begin{proof}

Sufficiency is a consequence of the inequality $\max \{ A_1, \cdots, A_9\}\leq c B$, while necessity follows immediately by taking the test function $f_{a_1, a_2}(x,y) = \chi_{B(e_1, a_1)}(x) \chi_{B(e_2, a_2)}(y)$, $a_1, a_2>0$.
\end{proof}

\vskip+0.2cm



\section*{Acknowledgements}
The first author is grateful to Professor V. Kokilashvili for drawing his attention to the two-weight problem for multiple Riesz potentials.

The first author was partially supported by  the Shota Rustaveli National
Science Foundation Grant (Contract Numbers:  D/13-23 and 31/47).


\end{document}